\documentclass[a4paper, 11pt, twoside, onecolumn ]{article}
\usepackage[utf8]{inputenc}
\usepackage{amsmath, amssymb, amsthm,mathrsfs}
\usepackage{enumerate} 
\usepackage{graphicx}
\usepackage[top=1in, bottom=1.3in, left=1.3in, right=1in]{geometry}

\usepackage{lmodern}
\usepackage[all]{xy}
\usepackage{tikz-cd} 
\usepackage{float}
\usepackage{hyperref}
\usepackage{accents}
\usepackage{verbatim}
\usepackage{import}
\usepackage{multirow}
\usepackage{multicol}
\DeclareMathOperator*{\esssup}{ess\,sup}
\date{}

\newtheorem{defn}{\bf Definition}[section] 

\newtheorem{lem}[defn]{\bf Lemma}
\newtheorem{thm}[defn]{\bf Theorem} 

\newtheorem{rmk}{\bf Remark}

\numberwithin{equation}{section}
\usepackage[pagewise]{lineno}

\newcommand{\footremember}[2]{%
	\footnote{#2}
	\newcounter{#1}
	\setcounter{#1}{\value{footnote}}%
}
\newcommand{\footrecall}[1]{%
	\footnotemark[\value{#1}]%
} 
\title{  Existence and regularity  results  for space-time fractional integro-differential equation of Kirchhoff type with memory  }
\author{%
	Lalit Kumar \footremember{alley}{Department of Mathematics, Indian Institute of Technology Bombay, Mumbai-400076, India. lalitccc528@gmail.com}%
	\and Sivaji Ganesh Sista\footremember{trailer}{Department of Mathematics, Indian Institute of Technology Bombay, Mumbai-400076, India. siva@math.iitb.ac.in}%
	\and Konijeti Sreenadh\footrecall{alley} \footnote{Department of Mathematics, Indian Institute of Technology Delhi - Abu Dhabi, Zayed City, Abu Dhabi UAE.  sreenadh@maths.iitd.ac.in}%
}

\begin{document}
	\maketitle

\begin{abstract}
    This paper analyses a Kirchhoff type quasilinear space-time fractional integro-differential equation with memory $(\mathcal{K}^{s}_{\alpha})$. Various a priori bounds are derived in different norms on the solution of  the considered equation. Utilizing these a priori bounds, existence and uniqueness of the weak solution to the proposed model are proved. Furthermore, regularity results on the solution of  $(\mathcal{K}^{s}_{\alpha})$ are established. The contribution made in this work provides a framework for  further investigation of such types of partial integro-differential equation $(\mathcal{K}^{s}_{\alpha})$. 
\end{abstract}
\section{Introduction}
The purpose of the current work  is to study the following space-time fractional Kirchhoff type diffusion equation  with memory term.
Let $\Omega \subset \mathbb{R}^{d}$ be an open and bounded domain with smooth boundary $\partial \Omega$ and $[0,T]$ be a fixed finite time interval. For each $(\alpha,s)\in (0,1)\times (0,1)$, find $u := u(x,t) : \Omega \times [0,T]\rightarrow \mathbb{R}$ that satisfies
\begin{equation*}\label{problem statement}\tag{$\mathcal{K}^{s}_{\alpha}$}
\begin{aligned}
    \partial^{\alpha}_{t}u+M\left(\|u(t)\|^{2}_{X_{0}(\Omega)}\right)(-\Delta)^{s}u&=f(x,t)+ \int_{0}^{t}b(x,t,\tau)u(\tau)~d\tau~\text{in}~\Omega \times (0,T],\\
    u(x,0)&=u_{0}(x)~~~\text{in}~~~\Omega,\\
    u(x,t)&=0~~~\text{in}~~~\Omega^{c}\times [0,T],
    \end{aligned}
\end{equation*}
where 
\begin{equation*}
\|u(t)\|^{2}_{X_{0}(\Omega)}=\int_{\left(\mathbb{R}^{d}\times \mathbb{R}^{d}\right)\backslash \left(\Omega^{c}\times \Omega^{c}\right)}\frac{(u(x,t)-u(y,t))^{2}}{|x-y|^{d+2s}}~dxdy.
\end{equation*}
The notation $\partial^{\alpha}_{t}u(x,t)$ in  \eqref{problem statement} refers to the regularized  Caputo time-fractional derivative  \cite{diethelm2002analysis, Kubica, webb2019weakly} that is given by 
\begin{equation}\label{Definition of fractional derivative}
    \partial^{\alpha}_{t}u(x,t)=\frac{d}{dt}\int_{0}^{t}k(t-\tau)(u(\tau)-u_{0})~d\tau,
\end{equation}
with
\begin{equation*}\label{15-5-16}
k(t)=\frac{t^{-\alpha}}{\Gamma(1-\alpha)}.
\end{equation*}
The fractional Laplacian operator \cite{mingqi2018nonlocal,di2012hitchhiker} in \eqref{problem statement} is defined as 
\begin{equation}\label{15-4-18}
    (-\Delta)^{s}u(x,t)= C_{d,s}~ P.~V.\int_{\mathbb{R}^{d}}\frac{u(x,t)-u(y,t)}{|x-y|^{d+2s}}~dy~~~~\forall~ (x,t) \in \Omega \times [0,T], 
\end{equation}
where $P.~V.$ denotes the principal value and $C_{d,s}=\frac{2^{2s}s\Gamma(s+\frac{d}{2})}{\pi^{d/2}\Gamma(1-s)}$ is the normalizing constant.
 Initial condition $u_{0}$, source function $f$, nonlocal diffusion coefficient $M$, and memory operator $b$ are known terms that are described in Section \ref{7-1-24-1}.
 \par The proposed model $\left(\mathcal{K}^{s}_{\alpha}\right)$ and its variants arise naturally in the modelling of various phenomena such as finance \cite{applebaum2004levy}, nuclear reactor dynamics \cite{barbeiro2011h1},  population dynamics \cite{chipot2003remarks}, image processing \cite{cuesta2012image}.  Partial differential equations (PDEs) based on time-fractional derivative describe the anomalous diffusion where particle spread at a rate inconsistent with Brownian motion \cite{del2004fractional}. Also, the process in which mean square displacement of the diffusing particles depends on the fractional power of time \cite{luchko2012anomalous} are interpreted through PDEs involving time-fractional derivative. The fractional Laplacian operator \eqref{15-4-18} appears frequently in  L{\'e}vy stable diffusion process and in anomalous diffusion in material science \cite{biler2015nonlocal}. One may refer to \cite{goel2022critical, rawat2021multiple, binlin2019existence} and references therein for more applications of fractional Laplacian and related works. 
 \par Kirchhoff type diffusion coefficient occurs in various physical and biological processes. For instance, diffusion of a bacteria in a jar where increase in population density subject to spreading of bacteria \cite{chipot2003remarks}, transversal oscillations in vibrating string by considering the change in length of the string during vibrations \cite{kirchhoff1883vorlesungen}. The memory effect comes into play when the physical process takes place in a non homogeneous medium. For example, heat conduction in materials with memory, theory of viscoelasticity \cite{miller1978integrodifferential}.
 \par Over the past few years, several research articles analyses Kirchhoff type diffusion  equations involving  fractional derivatives  due to its wide range of applications \cite{ding2020local,fu2022global,kumar2021finite,kumar2023linearized,shen2022time}. For the case $b=0,~\alpha \rightarrow 1,$ and $s \rightarrow 1$ in \eqref{problem statement}, authors in \cite{chipot2003remarks} studied the following equation 
 \begin{equation}\label{22-1-24-1}
 u_{t}-M\left(\|\nabla u(t)\|^{2}\right)\Delta u =f(x,t)\quad \text{in}~~\Omega \times (0,T].
 \end{equation}
 They obtain the existence, uniqueness, and asymptotic behaviour of the solution of \eqref{22-1-24-1}. In \cite{mingqi2018nonlocal}, authors proved the local existence of the weak solution and its blow up in finite time for the following semilinear version of \eqref{problem statement} with $b=0$ and $\alpha \rightarrow 1$
 \begin{equation}\label{22-1-24-2}
 u_{t}+M\left(\|u(t)\|^{2}_{X_{0}(\Omega)}\right)(-\Delta)^{s}u=|u|^{\rho}u \quad \text{in}~~\Omega \times (0,\infty),
 \end{equation}
 where $0 < \rho <\frac{4s}{d-2s}$ and $d>2s$. 
 \par Authors in \cite{kumar2020finite} considered the following model with $\alpha \rightarrow 1, ~s \rightarrow 1$ in \eqref{problem statement} 
 \begin{equation}\label{22-1-24-3}
 u_{t}-M\left(\|\nabla u(t)\|^{2}\right)\Delta u =f(x,t)+\int_{0}^{t}b(x,t,\tau)u(\tau)~d\tau\quad \text{in}~~\Omega \times (0,T],
 \end{equation}
 and proved the existence and uniqueness of the weak solution by Galerkin method. Further, authors in \cite{kumar2021finite} extended the equation \eqref{22-1-24-3} for $\alpha \in (0,1)$ as
 \begin{equation}\label{22-1-24-4}
 \partial^{\alpha}_{t}u-M\left(\|\nabla u(t)\|^{2}\right)\Delta u =f(x,t)+\int_{0}^{t}b(x,t,\tau)u(\tau)~d\tau\quad \text{in}~~\Omega \times (0,T].
 \end{equation}
 For this problem \eqref{22-1-24-4}, authors  applied Galerkin method to show the existence and uniqueness of the weak solution.
 \par In this work, the considered model \eqref{problem statement} contains a nonlinear diffusion coefficient due to which we do not have  explicit representation of the solution.  To handle this difficulty, we apply Galerkin method \cite{fu2022global, kumar2021finite} to prove the existence and uniqueness of the weak solution to the equation \eqref{problem statement}. In order to apply Galerkin method, we derive a priori bounds on the solution in various norms. These a priori bounds enable us to apply Aubin-Lions type compactness lemma for time-fractional derivative to ensure the convergence of the Galerkin sequence only in $L^{2}(0,T;L^{2}(\Omega))$ norm.
 \par Due to the presence of Kirchhoff term $M\left(\|u(t)\|^{2}_{X_{0}(\Omega)}\right)$  and memory operator in \eqref{problem statement} we require the convergence of this Galerkin sequence in $L^{2}(0,T;X_{0}(\Omega))$ norm. Thus the convergence in $L^{2}(0,T;L^{2}(\Omega))$ norm is not sufficient to conclude the  existence of the weak solution. Therefore, this paper establishes the convergence of Galerkin sequence in $L^{2}(0,T;X_{0}(\Omega))$ norm.
 \par Since it is difficult to have explicit representation of the solution to the nonlinear equation \eqref{problem statement} therefore, we cannot get the regularity results directly. To resolve this issue, we write the implicit form of solution by applying integral transformation on \eqref{problem statement}. Then using derived a priori bounds on the solution, we prove regularity results on the solution of \eqref{problem statement}. These results form a background for further investigation of problem under consideration. To the best of our knowledge, this is the first attempt in the literature which discusses the existence, uniqueness, and regularity of the weak solution to the equation \eqref{problem statement}.

 \par The organization of this paper is as follows. In Section \ref{7-1-24-1}, we present  the functional framework, some preliminaries results, and state the main results. Section \ref{7-1-24-2} contains the proof of existence and uniqueness of the weak solution to the equation \eqref{problem statement}. In Section \ref{7-1-24-3}, we derive regularity results on the solution of the equation \eqref{problem statement}. 
 
\section{Preliminaries and main results}\label{7-1-24-1}
In this section, we  define  some function spaces and assumptions on the given data that are needed throughout the analysis of \eqref{problem statement}. We also state the main results of this article.
\par Let $L^{2}(\Omega)$  be the space of square integrable functions on $\Omega$ associated with the norm $\|\cdot\|$ which is  induced by the inner product $(\cdot,\cdot)$. The fractional Sobolev space $H^{s}\left(\Omega\right)~(0<s<1)$ \cite{di2012hitchhiker} is defined as 
 \begin{equation}\label{15-4-7}
 H^{s}\left(\Omega\right)=\left\{u \in L^{2}\left(\Omega\right);~ |u|_{H^{s}(\Omega)} < \infty \right\},
 \end{equation}
 where
 \begin{equation}
 |u|_{H^{s}(\Omega)}=\left(\int_{\Omega\times \Omega}\frac{(u(x)-u(y))^{2}}{|x-y|^{d+2s}}~dxdy\right)^{\frac{1}{2}}
 \end{equation}
 denotes the Aronszajn-Slobodecki seminorm. The space  $H^{s}\left(\Omega\right)$ is a Hilbert space equipped with the inner product 
 \begin{equation}
 (u,v)_{H^{s}(\Omega)}=(u,v)+\int_{\Omega\times \Omega}\frac{(u(x)-u(y))(v(x)-v(y))}{|x-y|^{d+2s}}~dxdy.
 \end{equation}
 The classical fractional Sobolev space \eqref{15-4-7} approach is not sufficient to study the weak formulation of the problems involving fractional Laplacian operator \cite{servadei2013lewy}. To resolve this issue, authors in \cite{Giacomoni, rawat2021multiple,servadei2012mountain, servadei2013lewy} define the following space 
 
 \begin{equation}
 X(\Omega)=\left\{u\mid ~ u:\mathbb{R}^{d}\rightarrow \mathbb{R}~ \text{is measurable}, u|_{\Omega} \in L^{2}(\Omega) ~\text{and}~ \frac{(u(x)-u(y))}{|x-y|^{\frac{d+2s}{2}}} \in L^{2}(Q)\right\},
 \end{equation}
  where $Q=\left(\mathbb{R}^{d}\times \mathbb{R}^{d}\right)\backslash \left( \Omega^{c}\times \Omega^{c}\right)$. The norm on this space $X(\Omega)$ is given by 
  \begin{equation}\label{28-3-24-3}
  \|u\|_{X(\Omega)}=\|u\|+ \left(\int_{Q}\frac{(u(x)-u(y))^{2}}{|x-y|^{d+2s}}~dxdy\right)^{\frac{1}{2}}.
  \end{equation}
  Further, we define the space $X_{0}(\Omega)=\{u\mid~u\in X(\Omega)~\text{and}~u=0~a.e.~\text{in}~\Omega^{c}\}$ which is an Hilbert space \cite[Lemma 7]{servadei2012mountain}  with the  following inner product 
 \begin{equation}\label{15-4-10}
 \begin{aligned}
     (u,v)_{X_{0}(\Omega)} =\frac{C_{d,s}}{2}\int_{Q}\frac{(u(x)-u(y))(v(x)-v(y))}{|x-y|^{d+2s}}~dxdy\quad \forall~u,v~\in X_{0}(\Omega),
     \end{aligned}
 \end{equation} 
 where $C_{d,s}$ is defined in \eqref{15-4-18}. On this space $X_{0}(\Omega)$, the  norm  $\|\cdot\|_{X_{0}(\Omega)}$ is induced by the inner product \eqref{15-4-10} and equivalent to the norm $\|\cdot\|_{X(\Omega)}$ \cite[Lemma 6]{servadei2012mountain} defined in \eqref{28-3-24-3}.\\
\noindent  For any space $Z$, we define 
\begin{equation}
L^{2}(0,T;Z)=\left\{u\mid~u:[0,T] \rightarrow Z~\text{is measurable and}~ \int_{0}^{T}\|u(\tau)\|_{Z}^{2}~d\tau<\infty\right\},
\end{equation}
with the following norm 
 \begin{equation}\label{15-4-12}
     \|u\|_{L^{2}(0,T;Z)}=\left(\int_{0}^{T}\|u(\tau)\|_{Z}^{2}~d\tau\right)^{\frac{1}{2}}.
 \end{equation}
 Similarly, 
 \begin{equation}
L^{2}_{\alpha}(0,T;Z)=\left\{u\mid~u:[0,T] \rightarrow Z~\text{is measurable and}~ \sup_{t\in (0,T)} \left(\int_{0}^{t}(t-\tau)^{\alpha-1}\|u(\tau)\|_{Z}^{2}~d\tau\right)<\infty\right\}.
\end{equation}
The norm on this  space $L^{2}_{\alpha}(0,T;Z)$ is defined by 
 \begin{equation}\label{15-4-13}
     \|u\|_{L^{2}_{\alpha}(0,T;Z)}^{2}=\sup_{t\in (0,T)}\left(\int_{0}^{t}(t-\tau)^{\alpha-1}\|u(\tau)\|_{Z}^{2}~d\tau\right).
 \end{equation}
 The space $L^{\infty}(0,T;Z)$ contains the measurable functions $u:[0,T] \rightarrow Z$ such that
 \begin{equation}
 \esssup_{t\in(0,T)} \|u(t)\|_{Z} < \infty,
 \end{equation}
  and the norm  on this space is given by   
 \begin{equation}\label{15-4-14}
     \|u\|_{L^{\infty}(0,T;Z)}=\esssup_{t\in(0,T)} \|u(t)\|_{Z}.
 \end{equation}
\noindent  Further, denote $W^{1,1}[0, T]$ be the space of measurable functions $u:[0, T]\rightarrow \mathbb{R}$ such that $u$  is integrable (i.e., $u\in L^{1}[0, T]$) and $u_{t}\in L^{1}[0, T]$, here derivative is to be understood in the  sense of  distributions. The convolution between two integrable functions $g$ and $h$ is denoted by $"\ast"$ which means
\begin{equation}
(g\ast h)(t)=\int_{0}^{t}g(t-\tau)h(\tau)~d\tau\quad \forall~t\in [0,T].
\end{equation}
We assume the following hypotheses on the given data :\\
\textbf{(H1)}\quad Initial condition $u_{0} \in X_{0}(\Omega)$ and  source term $f \in L^{2}_{\alpha}(0,T;L^{2}(\Omega))$.\\
\textbf{(H2)}\quad The nonlocal diffusion coefficient $M:(0,\infty)\rightarrow \mathbb{R}$  is a Lipschitz continuous function with Lipschitz constant $L_{M}$ such that there exists a positive constant $m_{0}$ satisfying
 \begin{equation}\label{positivity of diffusion coefficient}
     M(\sigma)\geq m_{0}>0~\forall~\sigma \in (0,\infty)~\text{and}~m_{0}-4L_{M}K^{2}>0,
 \end{equation}
\noindent  where $K=\|u_{0}\|_{X_{0}(\Omega)}+\|f\|_{L^{2}_{\alpha}(0,T;L^{2}(\Omega))}$.\\
\textbf{(H3)} The memory operator $b(x,t,\tau)$ is given by 
 \begin{equation}\label{28-3-24-1}
 b(x,t,\tau)u(\tau)=\beta (-\Delta)^{s}u(\tau)+b_{0}(x,t,\tau)u(\tau)~~~\forall~ (x,t,\tau) \in \Omega \times [0,T] \times [0,T],
 \end{equation}
 where $\beta$ is a real parameter and $b_{0}:\Omega \times [0,T] \times [0,T] \rightarrow \mathbb{R}$ is a continuous function.
 \begin{rmk} In the diffusion term $(-\Delta)^{s}u$ and in the memory term $b(x,t,\tau)u(\tau)$ \eqref{28-3-24-1}, we can choose different values of $s\in (0,1)$. We just need the following estimate in our analysis
 \begin{equation}\label{28-3-24-4}
     \|b(x,t,\tau)u(\tau)\| \leq \|(-\Delta)^{s}u(\tau)\|\quad \forall~u\in X_{0}(\Omega),~~~\forall~(x,t,\tau) \in \Omega \times [0,T]\times[0,T].
 \end{equation}\label{28-3-24-2}
 For example, if we take diffusion term as $(-\Delta)^{s_{1}}u$ and memory term  $\beta (-\Delta)^{s_{2}}u(\tau)+b_{0}(x,t,\tau)u(\tau)$ then for $s_{1}\geq s_{2}$ the relation \eqref{28-3-24-4} holds \cite[Proposition 2.1]{di2012hitchhiker}.
 \end{rmk}
 \noindent For any two quantities $a$ and $b$, the notation $a\lesssim b$ means that there exists a generic positive constant $C$ (which may vary at different occurrences) such that $a \leq Cb$, where $C$ is independent of  $u$ but may depend on the given data.
 \par The following two Lemmas play a pivotal role in deriving a priori bounds on  the solution of the problem \eqref{problem statement}.
 \begin{lem}\cite[Lemma 18.4.1]{gripenberg1990volterra} \label{Positivity of fractional derivative}
Let $H$ be a Hilbert space and $T>0$.Then for any $ \tilde{k} \in W^{1,1}([0,T])$ and any $v\in L^{2}(0,T;H)$ there holds 
\begin{equation*}
\begin{aligned}
    \left(\frac{d}{dt}(\tilde{k}\ast v)(t),v(t)\right)_{H}&=~\frac{1}{2}\frac{d}{dt}\left(\tilde{k} \ast \|v(\cdot)\|^{2}_{H}\right)(t)+\frac{1}{2}\tilde{k}(t)\|v(t)\|_{H}^{2}\\
    &+\frac{1}{2}\int_{0}^{t}[-(\tilde{k})'(s)]\|v(t)-v(t-s)\|_{H}^{2}~ds~~\text{a.e.}~t~\in~(0,T).
    \end{aligned}
\end{equation*}
\end{lem}

\begin{lem}\label{Gronwall inequality}\cite{almeida2017gronwall}
Let $\alpha \in (0,1)$ and suppose that $u$ and $v$ are two nonnegative integrable functions on $[a,b]$, $v$ is nondecreasing and $g$ is a continuous function in $[a,b]$. If
\begin{equation*}
    u(t)\leq v(t)+g(t)\int_{0}^{t}(t-s)^{\alpha-1}u(s)ds~~\forall~ t \in [a,b],
\end{equation*}
then 
\begin{equation*}
    u(t)\leq v(t)E_{\alpha}\left[g(t)\Gamma(\alpha)t^{\alpha}\right]~\forall~t\in [a,b],
\end{equation*}
where $E_{\alpha}(\cdot)$ is the Mittag-Leffler function  defined in \cite{podlubny1998fractional} as $E_{\alpha}(z)=\sum_{j=0}^{\infty}\frac{z^{j}}{\Gamma(1+\alpha j)}.$
\end{lem}

\begin{lem}\cite[Theorem 1.6]{podlubny1998fractional}\label{8-1-24-2}
    Let $0 < \alpha <2$ and  $\mu$ is such that $\frac{\pi \alpha}{2} < \mu < \min\{\pi,\pi\alpha\}$. Then there exists a positive constant $C$ such that 
    \begin{equation}
    E_{\alpha}(z) \leq \frac{C}{1+|z|}\quad \text{for}~ \mu \leq  |arg(z)| \leq \pi.
    \end{equation}
    
\end{lem}

\begin{thm}   \label{13-5-1}\cite{diethelm2002analysis} Consider the following initial value problem 
\begin{equation}\label{13-5-2}
\begin{aligned}
    \partial^{\alpha}_{t}y(t)&=g(t,y(t)),~ t \in (0,T],~\alpha \in (0,1),\\
    y(0)&=y_{0}.
    \end{aligned}
\end{equation}
Let $y_{0}\in \mathbb{R}, K^{\ast}>0, t^{\ast}>0$. Define $D=\left\{(t,y(t));~t\in [0,t^{\ast}],~|y-y_{0}|\leq K^{\ast}\right\}$. Let function $g:D\rightarrow \mathbb{R}$ be a continuous. Define $M^{\ast}=\sup_{(t,y(t))\in D}|g(t,y(t)|.$  Then there exists a continuous function $y\in C[0,T^{\ast}]$ which solves the problem \eqref{13-5-2}, where 

\begin{equation}
T^{\ast}=\begin{cases} t^{\ast};&~~~~\text{if}~~M^{\ast}=0,\\
\min\{t^{\ast}, \left(\frac{K^{\ast}\Gamma(1+\alpha)}{M^{\ast}}\right)^{\frac{1}{\alpha}}\};&~~~~\text{if}~~M^{\ast}\neq 0.
\end{cases}
\end{equation}
\end{thm}

\begin{lem}\label{13-5-3} \cite[Theorem  4.1]{li2018some}
For $T>0$ and $\alpha \in (0,1)$. Let $X, Y,$ and $Z$ be the Banach spaces such that $X$ is compactly embedded in $Y$ and $Y$ is continuously embedded in $Z$. Suppose that $W\subset L^{1}_{\text{loc}}(0,T;X)$ satisfies the following 
\begin{enumerate}
    \item There exists a constant $C_{1}>0$ such that for all $u\in W$
    \begin{equation}
       \sup_{t \in (0,T)}\left(\frac{1}{\Gamma(\alpha)}\int_{0}^{t}(t-s)^{\alpha-1}\|u(s)\|^{2}_{X}~ds\right) \leq  C_{1}. 
    \end{equation}
    \item There exists a constant $C_{2}>0$ such that for all $u\in W$
    \begin{equation}
        \|\partial^{\alpha}_{t}u\|_{L^{2}(0,T;Z)}\leq C_{2}.
    \end{equation}
\end{enumerate}
Then $W$ is relatively compact in $L^{2}(0,T;Y)$.
\end{lem}

\begin{lem}\label{29-3-23-5}  \cite[Propoition 1.3]{pedregal1997parametrized}
A sequence of functions $\{f_{j}\}_{j=1}^{\infty}$ is relatively compact in $L^{1}[0,T]$ iff there exists a constant $C>0$ such that $\|f_{j}\|_{L^{1}(0,T)}\leq C~\forall~j$ and for every $\epsilon>0$ there exists a constant  $\delta=\delta(\epsilon)>0$ such that for all measurable subset E with $|E|<\delta$, we have 
\begin{equation}
\int_{E}|f_{j}(t)|~dt<\epsilon \quad \text{uniformly}~\forall~j.
\end{equation}
\end{lem}
The main results of this article are given below.
\begin{thm}
Suppose that (H1)-(H3) hold. Then there exists a unique weak solution $u$ \eqref{weak form 1} to the equation \eqref{problem statement} which satisfies the following a priori bounds 
\begin{align}
 \|u\|_{L^{\infty}(0,T;L^{2}(\Omega))}^{2}+\|u\|^{2}_{L^{2}_{\alpha}(0,T;X_{0}(\Omega))}&\lesssim \|u_{0}\|^{2}+\|f\|^{2}_{L^{2}_{\alpha}(0,T;L^{2}(\Omega))}.\\
 \|u\|_{L^{\infty}(0,T;X_{0}(\Omega))}^{2}+\|(-\Delta)^{s}u\|^{2}_{L^{2}_{\alpha}(0,T;L^{2}(\Omega))}&\lesssim \|u_{0}\|^{2}_{X_{0}(\Omega)}+\|f\|^{2}_{L^{2}_{\alpha}(0,T;L^{2}(\Omega))}.\\
 \|\partial^{\alpha}_{t}u\|^{2}_{L^{2}(0,T;L^{2}(\Omega))}
    &\lesssim \|u_{0}\|^{2}_{X_{0}(\Omega)}+\|f\|^{2}_{L^{2}_{\alpha}(0,T;L^{2}(\Omega))}.
\end{align}
\end{thm}
 \begin{thm}\label{29-3-23-22-1}  Under the assumptions (H1)-(H3),  we have the following regularity estimate on  the solution $u$ of the problem \eqref{problem statement}
 \begin{equation}\label{29-3-23-21-1-1}
 \|u\|_{L^{2}(0,T;H^{s+\nu}(\Omega))} \lesssim \|u_{0}\|_{X_{0}(\Omega)}+\|f\|_{L^{2}_{\alpha}(0,T;L^{2}(\Omega))},
 \end{equation}
  where $\nu=\min\left\{s,\frac{1}{2}-\epsilon\right\}$ for some small $\epsilon >0$.
 \end{thm}

 \begin{thm}\label{13-1-24-1-1} Suppose that (H1)-(H3) hold. Moreover, $(-\Delta)^{s}u_{0}\in L^{2}(\Omega)$ and $f(t)\in X_{0}(\Omega)$ such that  $\|(-\Delta)^{s}f(t)\|$ is bounded for every $t\in [0,T]$. Then, the solution $u$ of the problem  \eqref{problem statement} satisfies the following estimates 
\begin{align}
\|u(t)\|&\lesssim ~(1+t^{\alpha})\quad \forall~t\in [0,T] \label{12-1-24-16-1}.\\
\|u(t)-u(t^{\ast})\| &\lesssim ~|t-t^{\ast}|^{\alpha}\quad \forall~t,t^{\ast}\in [0,T]\label{12-1-24-15-1}.\\
\|\partial^{\alpha}_{t}u\|&\lesssim~ C\quad \forall~t\in [0,T]\label{12-1-24-15A-1}.
\end{align}
\end{thm}

\section{Existence and uniqueness of weak  solutions to $(\mathcal{K}^{s}_{\alpha})$}\label{7-1-24-2}
In this section, we prove the existence and uniqueness of the weak solution to the problem  \eqref{problem statement} by Galerkin method. In this method, first we consider the weak formulation \eqref{weak form 1}  of finding its solutions in a finite dimensional subspace of $X_{0}(\Omega)$. Then the existence Theorem \ref{13-5-1} for fractional differential equation  provides a Galerkin sequence of solutions of finite dimensional problems. We derive  various types of a priori bounds on the Galerkin sequence of solutions that help us to apply compactness Lemma \ref{13-5-3} to ensure the convergence in $L^{2}(0,T;L^{2}(\Omega))$. Further using Lemma \ref{29-3-23-5}, we prove that the Galerkin sequence also converges in $L^{2}(0,T;X_{0}(\Omega))$. As a consequence of these convergence results, we pass the limit inside the finite dimensional problems to   conclude the existence of the weak solutions to \eqref{problem statement}. Uniqueness of the weak solution is proved by using Lipschitz continuity  of the diffusion coefficient $M$.
\par The weak formulation corresponding to the problem  \eqref{problem statement} is to find  $u \in L^{\infty}(0,T;L^{2}(\Omega))\cap L^{2}(0,T;X_{0}(\Omega))$ and $\partial^{\alpha}_{t}u \in L^{2}(0,T;L^{2}(\Omega))$ 
such that the following equations hold for all $v \in X_{0}(\Omega)$ and a.e. $t \in (0,T]$ 
\begin{equation}\label{weak form 1}
\begin{aligned}
\left(\partial^{\alpha}_{t}u,v\right)+M\left(\|u(t)\|^{2}_{X_{0}(\Omega)}\right)(u,v)_{X_{0}(\Omega)}&=(f,v)+\int_{0}^{t}B(t,\tau,u(\tau),v)~d\tau\\
u(x,0)&=u_{0}(x)~~~\text{in}~~~\Omega,
    \end{aligned}
\end{equation}
where $B(t,\tau,u(\tau),v)$ is given by 
\begin{equation}
B(t,\tau,u(\tau),v)= \beta(u(\tau),v)_{X_{0}(\Omega)}+(b_{0}(x,t,\tau)u(\tau),v)\quad \forall~u,v \in X_{0}(\Omega)~\text{and}~\forall~t,\tau \in [0,T].
\end{equation}
Using the continuity of  $b_{0}(x,t,\tau)$, it can be easily verified that there exists a positive constant $B_{0}$ such that 
\begin{equation}\label{8-1-24-1}
|B(t,\tau,u(\tau),v)| \leq B_{0}\|u(\tau)\|_{X_{0}(\Omega)}\|v\|_{X_{0}(\Omega)}\quad \forall~u,v \in X_{0}(\Omega)~\text{and}~\forall~t,\tau \in [0,T].
\end{equation}
First, we prove the following a priori bounds  on the solution to the problem \eqref{problem statement}.
\begin{lem}\textbf{(A priori bound 1)}\label{21-4-9}
Suppose that  (H1)-(H3) hold. Then the solution $u$ of the problem \eqref{problem statement} satisfies the following a priori bound
\begin{equation}\label{20-4-1}
    \|u\|_{L^{\infty}(0,T;L^{2}(\Omega))}^{2}+\|u\|^{2}_{L^{2}_{\alpha}(0,T;X_{0}(\Omega))}\lesssim \|u_{0}\|^{2}+\|f\|^{2}_{L^{2}_{\alpha}(0,T;L^{2}(\Omega))}.
    \end{equation}
\end{lem}
\begin{proof}
Put $v=u(t)$ in \eqref{weak form 1}, we get 
\begin{equation}\label{5}
    \begin{aligned}
    \left(\partial^{\alpha}_{t}u,u\right)&+M\left(\|u(t)\|^{2}_{X_{0}(\Omega)}\right)\|u\|^{2}_{X_{0}(\Omega)}=(f,u)+\int_{0}^{t}B(t,\tau,u(\tau),u(t))~d\tau.
    \end{aligned}
\end{equation}
By the definition of fractional derivative \eqref{Definition of fractional derivative}, the equation \eqref{5} is rewritten as 
\begin{equation}\label{6}
    \begin{aligned}
    \left(\frac{d}{dt}(k\ast (u-u_{0}))(t),u(t)\right)+M\left(\|u(t)\|^{2}_{X_{0}(\Omega)}\right)\|u\|^{2}_{X_{0}(\Omega)}
    &=(f,u)+\int_{0}^{t}B(t,\tau,u(\tau),u(t))~d\tau.
    \end{aligned}
\end{equation}
\noindent In the equation  \eqref{6},  $k$  does not belong to $W^{1,1}[0,T]$ so  we cannot apply the Lemma \ref{Positivity of fractional derivative} directly in the equation \eqref{6}. To make use of the Lemma \ref{Positivity of fractional derivative}, we approximate $\frac{d}{dt}(k\ast u)$ by its Yosida approximation $\frac{d}{dt}(k_{p}\ast u),p\in \mathbb{N}$ \cite{zacher2009weak} such that 
\begin{equation}
\frac{d}{dt}(k_{p}\ast u) \rightarrow \frac{d}{dt}(k\ast u) \quad \text{in}~~L^{2}(0,T;L^{2}(\Omega)) \quad \text{as}~p \rightarrow \infty,
\end{equation}
where the  kernels $k_{p},p\in \mathbb{N}$ are nonnegative, nonincreasing in $(0,\infty) $ and belong to $W^{1,1}[0,T]$ with the property that $k_{p}\rightarrow k$ in $L^{1}[0,T]$ as $p \rightarrow \infty$. Thus, adding and subtracting $\left(\frac{d}{dt}(k_{p}\ast (u-u_{0}))(t),u(t)\right)$ in the equation \eqref{6}, we get 
\begin{equation}\label{9}
    \begin{aligned}
    \left(\frac{d}{dt}(k_{p}\ast u)(t),u(t)\right)+M\left(\|u(t)\|^{2}_{X_{0}(\Omega)}\right)\|u\|^{2}_{X_{0}(\Omega)}
    &=\left(R_{p}(t),u(t)\right)+\left(\frac{d}{dt}(k_{p}\ast u_{0})(t),u(t)\right)\\
    &+(f,u)+\int_{0}^{t}B(t,\tau,u(\tau),u(t))~d\tau,
    \end{aligned}
\end{equation}
where $R_{p}(t):=\left(\frac{d}{dt}(k_{p}\ast (u-u_{0}))(t)-\frac{d}{dt}(k\ast (u-u_{0}))(t)\right)$. Now $k_{p}, p \in \mathbb{N}$ belong to $W^{1,1}[0,T]$, therefore using Lemma \ref{Positivity of fractional derivative} the equation \eqref{9} is converted into 
    \begin{equation*}\label{10}
    \begin{aligned}
    &\frac{1}{2}\frac{d}{dt}\left(k_{p} \ast \|u\|^{2}\right)(t)+\frac{1}{2}k_{p}(t)\|u\|^{2}
    +\frac{1}{2}\int_{0}^{t}[-k_{p}'(s)]\|u(t)-u(t-s)\|^{2}~ds\\
    &+M\left(\|u(t)\|^{2}_{X_{0}(\Omega)}\right)\|u\|^{2}_{X_{0}(\Omega)}\\
    &=\left(R_{p}(t),u(t)\right)+\left(\frac{d}{dt}(k_{p}\ast u_{0})(t),u(t)\right)+(f,u)+\int_{0}^{t}B(t,\tau,u(\tau),u(t))~d\tau.
    \end{aligned}
\end{equation*}
Positivity of diffusion coefficient \eqref{positivity of diffusion coefficient}, monotonicity of $k_{p}$  and estimate on memory operator \eqref{8-1-24-1} imply  
\begin{equation*}\label{10-1}
    \begin{aligned}
    \frac{1}{2}\frac{d}{dt}\left(k_{p} \ast \|u\|^{2}\right)(t)+\frac{1}{2}k_{p}(t)\|u\|^{2}
    +m_{0}\|u\|^{2}_{X_{0}(\Omega)}&\leq \left(R_{p}(t),u(t)\right)+\left(\frac{d}{dt}(k_{p}\ast u_{0})(t),u(t)\right)\\
    &+(f,u)+B_{0}\int_{0}^{t}\|u(\tau)\|_{X_{0}(\Omega)}\|u(t)\|_{X_{0}(\Omega)}~d\tau.
    \end{aligned}
\end{equation*}
Evaluate $\frac{d}{dt}(k_{p}\ast u_{0})(t)$ and apply Cauchy-Schwarz and Young's inequality to get  
\begin{equation}\label{13}
    \begin{aligned}
    \frac{d}{dt}\left(k_{p} \ast \|u\|^{2}\right)(t)+k_{p}(t)\|u\|^{2}+m_{0}\|u\|^{2}_{X_{0}(\Omega)}&\leq \|R_{p}(t)\|^{2}+k_{p}(t)\|u_{0}\|^{2}+k_{p}(t)\|u\|^{2}+\|f\|^{2}\\
    &+2\|u\|^{2}+\frac{B_{0}^{2}T}{m_{0}}\int_{0}^{t}\|u(\tau)\|^{2}_{X_{0}(\Omega)}~d\tau.
    \end{aligned}
\end{equation}
Convolving  equation \eqref{13} with the kernel $l(t)=\frac{t^{\alpha-1}}{\Gamma(\alpha)}$ and using $l \ast\frac{d}{dt}\left(k_{p}\ast \|u\|^{2}\right)(t)=\frac{d}{dt}\left(k_{p}\ast l \ast \|u\|^{2}\right)(t)$ we have 
\begin{equation}\label{19}
    \begin{aligned}
    \frac{d}{dt}\left(k_{p}\ast l \ast \|u\|^{2}\right)(t)&
    +m_{0}\left(l\ast\|u\|^{2}_{X_{0}(\Omega)}\right)(t)\\
    &\leq \left(l\ast \|R_{p}\|^{2}\right)(t)+ \|u_{0}\|^{2} \left(l\ast k_{p}\right)(t)+\left(l \ast \|f\|^{2}\right)(t)\\
    &+2\left(l \ast \|u\|^{2}\right)(t)+\frac{B_{0}^{2}T}{m_{0}}\left(l \ast \int_{0}^{t}\|u(\tau)\|^{2}_{X_{0}(\Omega)}~d\tau\right)(t).
    \end{aligned}
\end{equation}
Let $p \rightarrow \infty$ in the equation \eqref{19}, then by using the  convergence properties of $k_{p}$ and $R_{p}$ we obtain 
\begin{equation}\label{23}
    \begin{aligned}
    \frac{d}{dt}\left(k\ast l \ast \|u\|^{2}\right)(t)
    +m_{0}\left(l\ast\|u\|^{2}_{X_{0}(\Omega)}\right)(t)
    &\leq  \|u_{0}\|^{2} \left(l\ast k\right)(t)+\left(l \ast \|f\|^{2}\right)(t)+2\left(l \ast \|u\|^{2}\right)(t)\\
    &+\frac{B_{0}^{2}T}{m_{0}}\left(l \ast \int_{0}^{t}\|u(\tau)\|^{2}_{X_{0}(\Omega)}~d\tau\right)(t).
    \end{aligned}
\end{equation}
Using the fact that $(k\ast l)(t)=1$ in the equation \eqref{23}, we conclude
\begin{equation}\label{24}
    \begin{aligned}
    \|u(t)\|^{2} +m_{0}\left(l\ast\|u\|^{2}_{X_{0}(\Omega)}\right)(t)
    &\leq  \|u_{0}\|^{2}+\left(l \ast \|f\|^{2}\right)(t)
    +2\left(l \ast \|u\|^{2}\right)(t)\\
    &+\frac{B_{0}^{2}T}{m_{0}}\left(l \ast \int_{0}^{t}\|u(\tau)\|^{2}_{X_{0}(\Omega)}~d\tau\right)(t).
    \end{aligned}
\end{equation}
The equation \eqref{24} is rewritten as 
\begin{equation}\label{24-1}
    \begin{aligned}
    &\min\{1,m_{0}\}\left(\|u(t)\|^{2} +\left(l\ast\|u\|^{2}_{X_{0}(\Omega)}\right)(t)\right)\\
    &\leq  \|u_{0}\|^{2}+\left(l \ast \|f\|^{2}\right)(t)
    +\max\left\{2,\frac{B_{0}^{2}T}{m_{0}}\right\}\left(l \ast\left(\|u\|^{2}+ \int_{0}^{t}\|u(\tau)\|^{2}_{X_{0}(\Omega)}~d\tau\right)\right)(t).
    \end{aligned}
\end{equation}
Denote $\tilde{u}(t):=\|u(t)\|^{2}+\left(l\ast\|u\|^{2}_{X_{0}(\Omega)}\right)(t)$ and $\tilde{v}(t):=\|u_{0}\|^{2}+\left(l \ast \|f\|^{2}\right)(t)$ which reduce  the equation \eqref{24-1} into
\begin{equation*}\label{25}
    \begin{aligned}
    \tilde{u}(t)
    &\lesssim \tilde{v}(t)+\left(l \ast\left(\|u\|^{2}+ \int_{0}^{t}\|u(\tau)\|^{2}_{X_{0}(\Omega)}~d\tau\right)\right)(t)\\
    &\lesssim \tilde{v}(t)+\left(l \ast\left(\|u\|^{2}+ \int_{0}^{t}(t-\tau)^{1-\alpha}(t-\tau)^{\alpha-1}\|u(\tau)\|^{2}_{X_{0}(\Omega)}~d\tau\right)\right)(t)\\
    &\lesssim \tilde{v}(t)+\left(l \ast\left(\|u\|^{2}+ \int_{0}^{t}(t)^{1-\alpha}(t-\tau)^{\alpha-1}\|u(\tau)\|^{2}_{X_{0}(\Omega)}~d\tau\right)\right)(t)\\
    &\lesssim \tilde{v}(t)+\int_{0}^{t}(t-s)^{\alpha-1}\tilde{u}(s)~ds.
    \end{aligned}
\end{equation*}
$\tilde{v}(t)$ is nondecreasing on $(0,T)$ so we can apply Lemma \ref{Gronwall inequality} and Lemma \ref{8-1-24-2} to conclude 
\begin{equation}\label{26}
    \tilde{u}(t)\lesssim \tilde{v}(t)E_{\alpha}\left[\Gamma(\alpha)t^{\alpha}\right]\lesssim \tilde{v}(t)~~~~\forall~t\in [0,T].
\end{equation}
Hence the result \eqref{20-4-1} follows. 
\end{proof}

\begin{lem}\textbf{(A priori bound 2)}\label{21-4-8}
Under the assumptions (H1)-(H3), the solution $u$ of the equation \eqref{problem statement} satisfies the following a priori bounds
\begin{equation}\label{20-4-6}
    \|u\|_{L^{\infty}(0,T;X_{0}(\Omega))}^{2}+\|(-\Delta)^{s}u\|^{2}_{L^{2}_{\alpha}(0,T;L^{2}(\Omega))}\lesssim \|u_{0}\|^{2}_{X_{0}(\Omega)}+\|f\|^{2}_{L^{2}_{\alpha}(0,T;L^{2}(\Omega))},
    \end{equation}
    and
    \begin{equation}\label{32-5y}
    \begin{aligned}
  \|(-\Delta)^{s}u\|_{L^{2}(0,T;L^{2}(\Omega))}^{2} &\lesssim \|u_{0}\|^{2}_{X_{0}(\Omega)}+\|f\|^{2}_{L^{2}_{\alpha}(0,T;L^{2}(\Omega))}.
    \end{aligned}
\end{equation}
\end{lem}
\begin{proof}
Substitute $v=(-\Delta)^{s}u(t)$ in \eqref{weak form 1}, we obtain
\begin{equation}\label{30}
    \begin{aligned}
    \left(\partial^{\alpha}_{t}u,(-\Delta)^{s}u\right)+M\left(\|u(t)\|^{2}_{X_{0}(\Omega)}\right)\|(-\Delta)^{s}u\|^{2}
    &=(f,(-\Delta)^{s}u)\\
    &+\int_{0}^{t}\left(b(x,t,\tau)u(\tau),(-\Delta)^{s}u(t)\right)~d\tau.
    \end{aligned}
\end{equation}
Again using the Yosida approximation, we  rewrite equation \eqref{30} as 
\begin{equation}\label{32-1}
    \begin{aligned}
    \left(\frac{d}{dt}(k_{p}\ast u)(t),u(t)\right)_{X_{0}(\Omega)}&+M\left(\|u(t)\|^{2}_{X_{0}(\Omega)}\right)\|(-\Delta)^{s}u\|^{2}\\
    &=\left(R_{p}(t),(-\Delta)^{s}u(t)\right)+\left(\frac{d}{dt}(k_{p}\ast u_{0})(t),u(t)\right)_{X_{0}(\Omega)}\\
    &+(f,(-\Delta)^{s}u)+\int_{0}^{t}\left(b(x,t,\tau)u(\tau),(-\Delta)^{s}u(t)\right)~d\tau,
    \end{aligned}
\end{equation}
where $R_{p}(t):=\left(\frac{d}{dt}(k_{p}\ast (u-u_{0}))(t)-\frac{d}{dt}(k\ast (u-u_{0}))(t)\right)$. Positivity of the diffusion coefficient \eqref{positivity of diffusion coefficient}, estimate on memory operator \eqref{8-1-24-1}, Cauchy-Schwarz inequality and Young's inequality yield
\begin{equation}\label{32}
    \begin{aligned}
    \left(\frac{d}{dt}(k_{p}\ast u)(t),u(t)\right)_{X_{0}(\Omega)}&+\frac{m_{0}}{2}\|(-\Delta)^{s}u\|^{2}\\
    &\leq \frac{3}{2m_{0}}\|R_{p}(t)\|^{2} + \frac{1}{2}k_{p}(t)\|u_{0}\|^{2}_{X_{0}(\Omega)}+\frac{1}{2}k_{p}(t)\|u\|^{2}_{X_{0}(\Omega)}+\frac{3}{2m_{0}}\|f\|^{2}\\
    &+\frac{3}{m_{0}}T|\beta|^{2}\int_{0}^{t}\|(-\Delta)^{s}u(\tau)\|^{2}~d\tau+\frac{3}{m_{0}}T|B_{0}|^{2}\int_{0}^{t}\|u(\tau)\|^{2}~d\tau.
    \end{aligned}
\end{equation}
Apply Lemma \ref{Positivity of fractional derivative} to get 
\begin{equation}\label{32-2}
    \begin{aligned}
  \frac{1}{2}\frac{d}{dt}\left(k_{p} \ast \|u\|_{X_{0}(\Omega)}^{2}\right)(t)&+\frac{1}{2}k_{p}(t)\|u\|_{X_{0}(\Omega)}^{2}+\frac{m_{0}}{2}\|(-\Delta)^{s}u\|^{2}\\
    &\leq \frac{3}{2m_{0}}\|R_{p}(t)\|^{2} + \frac{1}{2}k_{p}(t)\|u_{0}\|^{2}_{X_{0}(\Omega)}+\frac{1}{2}k_{p}(t)\|u\|^{2}_{X_{0}(\Omega)}+\frac{3}{2m_{0}}\|f\|^{2}\\
    &+\frac{3}{m_{0}}T|\beta|^{2}\int_{0}^{t}\|(-\Delta)^{s}u(\tau)\|^{2}~d\tau+\frac{3}{m_{0}}T|B_{0}|^{2}\int_{0}^{t}\|u(\tau)\|^{2}~d\tau.
    \end{aligned}
\end{equation}
Proceeding further along the same lines of the proof of estimate \eqref{20-4-1}  to conclude the desired estimate \eqref{20-4-6}.
\par Moreover, integrating \eqref{32-2} on [0,t] we obtain 
\begin{equation}\label{32-3}
    \begin{aligned}
 & \frac{1}{2}\left(k_{p} \ast \|u\|_{X_{0}(\Omega)}^{2}\right)(t)+\frac{m_{0}}{2}\int_{0}^{t}\|(-\Delta)^{s}u(\tau)\|^{2}~d\tau\\
    &\leq \frac{3}{2m_{0}}\int_{0}^{t}\|R_{p}(\tau)\|^{2}~d\tau + \frac{1}{2}\int_{0}^{t}k_{p}(\tau)\|u_{0}\|^{2}_{X_{0}(\Omega)}~d\tau+\frac{3}{2m_{0}}\int_{0}^{t}\|f(\tau)\|^{2}~d\tau\\
    &+\frac{3}{m_{0}}T|\beta|^{2}\int_{0}^{t}\int_{0}^{\tau}\|(-\Delta)^{s}u(\omega)\|^{2}~d\omega d\tau+\frac{3}{m_{0}}T|B_{0}|^{2}\int_{0}^{t}\int_{0}^{\tau}\|u(\omega)\|^{2}~d\omega d\tau.
    \end{aligned}
\end{equation}
Drop the first term from LHS of \eqref{32-3} and letting $p$ goes to infinity we have
\begin{equation}\label{32-4}
    \begin{aligned}
  \int_{0}^{t}\|(-\Delta)^{s}u(\tau)\|^{2}~d\tau &\lesssim \|u_{0}\|^{2}_{X_{0}(\Omega)}+\int_{0}^{t}\|f(\tau)\|^{2}~d\tau+\int_{0}^{t}\int_{0}^{\tau}\|(-\Delta)^{s}u(\omega)\|^{2}~d\omega d\tau\\
    &+\int_{0}^{t}\int_{0}^{\tau}\|u(\omega)\|^{2}~d\omega d\tau.
    \end{aligned}
\end{equation}
Apply estimate \eqref{20-4-1} and classical discrete Gronwall's inequality  to conclude 
\begin{equation}\label{32-5}
    \begin{aligned}
  \|(-\Delta)^{s}u\|_{L^{2}(0,T;L^{2}(\Omega))}^{2} &\lesssim \|u_{0}\|^{2}_{X_{0}(\Omega)}+\|f\|^{2}_{L^{2}_{\alpha}(0,T;L^{2}(\Omega))}.
    \end{aligned}
\end{equation}
\end{proof}

\begin{lem}\textbf{(A priori bound 3)}\label{21-4-7}
Under the assumptions of Lemma  \ref{21-4-8}, the solution $u$ of the equation \eqref{problem statement} satisfies the following a priori bound
\begin{equation}\label{20-4-7}
   \|\partial^{\alpha}_{t}u\|^{2}_{L^{2}(0,T;L^{2}(\Omega))}
    \lesssim \|u_{0}\|^{2}_{X_{0}(\Omega)}+\|f\|^{2}_{L^{2}_{\alpha}(0,T;L^{2}(\Omega))}.
    \end{equation}
\end{lem}
\begin{proof}
Let $v \in L^{2}(\Omega)$ be an arbitrary element. Then multiply  \eqref{problem statement}  by  $v$ and integrate  over $\Omega$, we get 
\begin{equation*}\label{34}
    \begin{aligned}
    \left(\partial^{\alpha}_{t}u,v\right)=-M\left(\|u(t)\|^{2}_{X_{0}(\Omega)}\right)\left((-\Delta)^{s}u,v\right)+(f,v)+\int_{0}^{t}(b(x,t,\tau)u(\tau),v)~d\tau.
    \end{aligned}
\end{equation*}
Using estimate \eqref{20-4-6}, continuity of $M$, and  Cauchy-Schwarz inequality we have
\begin{equation}\label{34-1}
    \begin{aligned}
    \|\partial^{\alpha}_{t}u\|=\sup_{0\neq v \in L^{2}(\Omega)}\frac{|\left(\partial^{\alpha}_{t}u,v\right)|}{\|v\|}\lesssim \left(\|(-\Delta)^{s}u\|+\|f\|+\int_{0}^{t}\left(\|(-\Delta)^{s}u(\tau)\|+\|u(\tau)\|\right)~d\tau\right).
    \end{aligned}
\end{equation}
Squaring and integrating \eqref{34-1} on (0,T) and in the view of estimates \eqref{20-4-1}, \eqref{32-5} we conclude
\begin{equation*}\label{36}
    \begin{aligned}
    \|\partial^{\alpha}_{t}u\|^{2}_{L^{2}(0,T;L^{2}(\Omega))} \lesssim \|u_{0}\|^{2}_{X_{0}(\Omega)}+\|f\|^{2}_{L^{2}_{\alpha}(0,T;L^{2}(\Omega))}.
    \end{aligned}
\end{equation*}
\end{proof}

 \begin{thm}\label{30-3-23-1}
 Suppose that  (H1)-(H3) hold. Then there exists a unique weak solution to the problem \eqref{problem statement}. 
 \end{thm}
 \begin{proof} \textbf{(Existence)}
 Let $\{\lambda_{i},\phi_{i}\}_{i\in \mathbb{N}}$ be the eigenpairs of the problem \cite[Proposition 9]{servadei2013variational}
 \begin{equation}\label{20-1-24-1}
 \begin{aligned}
 (-\Delta)^{s}w&=\lambda w~~\text{in}~\Omega,\\
 w&=0~~\text{in}~\Omega^{c}.
 \end{aligned}
 \end{equation}
 Then eignevalues  $\{\lambda_{i}\}_{i\in \mathbb{N}}$ satisfy $0< \lambda_{1}\leq \lambda_{2}\leq \lambda_{3}\leq \dots$ and $\lambda_{k}\rightarrow \infty $ as $k \rightarrow \infty$. The set of eigenfunctions $\{\phi_{i}\}_{i\in \mathbb{N}}$  is an orthonormal basis of $L^{2}(\Omega)$ and orthogonal basis of $X_{0}(\Omega)$. For any fixed positive integer $m$, consider  a finite dimensional subspace $V_{m}$ of $X_{0}(\Omega)$ such that $V_{m}$ is spanned by $\{\phi_{i}\}_{i=1}^{m}$. We consider the weak formulation \eqref{weak form 1} onto this finite dimensional subspace $V_{m}$ as to  find $u_{m}(t)\in V_{m}$ which satisfies the   following equations for all $v_{m} \in V_{m}$ and $a.e.~t\in (0,T]$ 
\begin{equation}\label{N24a}
\begin{aligned}
\left(\partial^{\alpha}_{t}u_{m},v_{m}\right)+M\left(\|u_{m}(t)\|^{2}_{X_{0}(\Omega)}\right)(u_{m},v_{m})_{X_{0}(\Omega)}&=(f,v_{m})+\int_{0}^{t}B(t,\tau,u_{m}(\tau),v_{m})~d\tau\\
u_{m}(0)&=\sum_{i=1}^{m}(u_{0},\phi_{i})\phi_{i}.
\end{aligned}
\end{equation}
Put the identification  $u_{m}(x,t)=\sum_{i=1}^{m}\alpha_{i}(t)\phi_{i}(x)$ in equation \eqref{N24a}. Then using the properties of eigenfunctions $\{\phi_{i}\}_{i\in \mathbb{N}}$, the problem \eqref{N24a} is converted into a coupled system of nonlinear fractional differential equations. By  Lemma \ref{13-5-1}, there exists a unique continuous solution $u_{m}$ of \eqref{N24a} in $[0, T_{m}]~(0<T_{m}<T)$ such that $(k\ast(u_{m}-u_{m}(0)))$ has vanishing trace at $t=0$. These solutions are made globally on $[0,T]$  using the a priori bounds \eqref{20-4-1}, \eqref{20-4-6}, and \eqref{20-4-7} \cite[Page 1311]{fu2022global}. 
\par Taking $X=X_{0}(\Omega)$, $Y=L^{2}(\Omega)$, and $Z=L^{2}(\Omega)$ in Lemma \ref{13-5-3} and using a priori bounds \eqref{20-4-1} and \eqref{20-4-7}, we conclude 
\begin{equation}\label{29-3-23-1-2}
u_{m}\rightarrow u \quad \text{in}~L^{2}(0,T;L^{2}(\Omega))\quad \text{as}~m\rightarrow \infty.
\end{equation}
Thus $u_{m}\rightarrow u ~ a.e. ~\text{in}~\Omega \times(0,T).$ Now, utilizing a priori bounds \eqref{20-4-6}, \eqref{32-5}, \eqref{20-4-7}, and continuity of $M$, we have
\begin{align}
u_{m} & \overset{\ast}{\rightharpoonup} u\quad \text{in}~L^{\infty}(0,T;X_{0}(\Omega))\quad \text{as}~m\rightarrow \infty\label{29-3-23-2}\\
u_{m}~ & {\rightharpoonup}~ u\quad \text{in}~L^{2}(0,T;X_{0}(\Omega))\quad \text{as}~m\rightarrow \infty\label{29-3-23-2-1}\\
(-\Delta)^{s}u_{m}& \rightharpoonup(-\Delta)^{s}u \quad \text{in}~L^{2}(0,T;L^{2}(\Omega))\quad \text{as}~m\rightarrow \infty\label{9-1-24-1}\\
\partial^{\alpha}_{t}u_{m} & \rightharpoonup \partial^{\alpha}_{t}u\quad \text{in}~L^{2}(0,T;L^{2}(\Omega))\quad \text{as}~m\rightarrow \infty\label{29-3-23-1}\\
M\left(\|u_{m}(t)\|^{2}_{X_{0}(\Omega)}\right)u_{m} &\overset{\ast}{\rightharpoonup} \eta \quad \text{in}~L^{\infty}(0,T;X_{0}(\Omega))\quad \text{as}~m\rightarrow \infty.\label{29-3-23-3}
\end{align}
Let $v\in L^{2}(0,T;X_{0}(\Omega))$ be an arbitrary element. Then there exists a sequence $v_{m}\in V_{m}$ such that 
\begin{equation}\label{29-3-23-4}
v_{m}\rightarrow v\quad \text{in}~L^{2}(0,T;X_{0}(\Omega))\quad \text{as}~m\rightarrow \infty.
\end{equation}
Let $m \rightarrow  \infty $ in \eqref{N24a} and apply \eqref{29-3-23-2-1}-\eqref{29-3-23-4} to obtain
\begin{equation}\label{N24a1}
\begin{aligned}
\left(\partial^{\alpha}_{t}u,v\right)+(\eta,v)_{X_{0}(\Omega)}=(f,v)+\int_{0}^{t}B(t,\tau,u(\tau),v)~d\tau.
\end{aligned}
\end{equation}
To show the existence of the weak solution, we need to prove $M\left(\|u(t)\|^{2}_{X_{0}(\Omega)}\right)u=\eta$. It is enough to prove that $u_{m}\rightarrow u ~\text{in} ~X_{0}(\Omega) ~a.e. ~t \in (0,T)$. Using \eqref{20-4-6},  continuity of $M$, and Lemma \ref{29-3-23-5},  we deduce that $\left\{M\left(\|u_{m}(t)\|^{2}_{X_{0}(\Omega)}\right)\right\}$ is relatively compact in $L^{1}(0,T)$. Therefore, up to a subsequence we have 
\begin{equation}\label{29-3-23-6}
M\left(\|u_{m}(t)\|^{2}_{X_{0}(\Omega)}\right) \rightarrow \zeta(t)\quad \text{in}~L^{1}(0,T)~~\text{as}~m\rightarrow \infty.
\end{equation}
From \eqref{29-3-23-2} and  \eqref{29-3-23-6}, we conclude 
\begin{equation}
\int_{0}^{T}M\left(\|u_{m}(t)\|^{2}_{X_{0}(\Omega)}\right)\left(u,u_{m}-u\right)_{X_{0}(\Omega)}~dt\rightarrow 0~~\text{as}~m\rightarrow \infty.
\end{equation}
Hence 
\begin{equation}\label{29-3-23-7-1}
M\left(\|u_{m}(t)\|^{2}_{X_{0}(\Omega)}\right)\left(u,u_{m}-u\right)_{X_{0}(\Omega)}\rightarrow 0~~a.e.~t\in (0,T)~~\text{as}~m\rightarrow \infty.
\end{equation}
Consider 
\begin{equation}\label{29-3-23-7}
\begin{aligned}
&M\left(\|u_{m}(t)\|^{2}_{X_{0}(\Omega)}\right)\|u_{m}-u\|^{2}_{X_{0}(\Omega)}\\
&=M\left(\|u_{m}(t)\|^{2}_{X_{0}(\Omega)}\right)\|u_{m}-u\|^{2}_{X_{0}(\Omega)}-\int_{0}^{t}B(t,\tau,u_{m}(\tau),u_{m}-u)~d\tau\\
&+\int_{0}^{t}B(t,\tau,u_{m}(\tau),u_{m}-u)~d\tau.
\end{aligned}
\end{equation}
This can be rewritten as 
\begin{equation}\label{29-3-23-8}
\begin{aligned}
&M\left(\|u_{m}(t)\|^{2}_{X_{0}(\Omega)}\right)\|u_{m}-u\|^{2}_{X_{0}(\Omega)}\\
&=-M\left(\|u_{m}(t)\|^{2}_{X_{0}(\Omega)}\right)\left(u_{m},u\right)_{X_{0}(\Omega)}-M\left(\|u_{m}(t)\|^{2}_{X_{0}(\Omega)}\right)\left(u,u_{m}-u\right)_{X_{0}(\Omega)}\\
&+M\left(\|u_{m}(t)\|^{2}_{X_{0}(\Omega)}\right)\left(u_{m},u_{m}\right)_{X_{0}(\Omega)}-\int_{0}^{t}B(t,\tau,u_{m}(\tau),u_{m}-u)~d\tau\\
&+\int_{0}^{t}B(t,\tau,u_{m}(\tau),u_{m}-u)~d\tau.
\end{aligned}
\end{equation}
Put $v_{m}=u_{m}$ in equation \eqref{N24a} and substitute  in \eqref{29-3-23-8}, we get 
\begin{equation}\label{29-3-23-9}
\begin{aligned}
&M\left(\|u_{m}(t)\|^{2}_{X_{0}(\Omega)}\right)\|u_{m}-u\|^{2}_{X_{0}(\Omega)}\\
&=-M\left(\|u_{m}(t)\|^{2}_{X_{0}(\Omega)}\right)\left(u_{m},u\right)_{X_{0}(\Omega)}-M\left(\|u_{m}(t)\|^{2}_{X_{0}(\Omega)}\right)\left(u,u_{m}-u\right)_{X_{0}(\Omega)}-\left(\partial^{\alpha}_{t}u_{m},u_{m}\right)\\
&+(f,u_{m})+\int_{0}^{t}B(t,\tau,u_{m}(\tau),u)~d\tau+\int_{0}^{t}B(t,\tau,u_{m}(\tau),u_{m}-u)~d\tau.
\end{aligned}
\end{equation}
Denote 
\begin{equation}\label{10-1-24-1}
\begin{aligned}
\mathcal{Z}_{m}&=-M\left(\|u_{m}(t)\|^{2}_{X_{0}(\Omega)}\right)\left(u_{m},u\right)_{X_{0}(\Omega)}-M\left(\|u_{m}(t)\|^{2}_{X_{0}(\Omega)}\right)\left(u,u_{m}-u\right)_{X_{0}(\Omega)}-\left(\partial^{\alpha}_{t}u_{m},u_{m}\right)\\
&+(f,u_{m})+\int_{0}^{t}B(t,\tau,u_{m}(\tau),u)~d\tau.
\end{aligned}
\end{equation}
Let $m \rightarrow \infty$ in \eqref{10-1-24-1} and apply \eqref{29-3-23-3}, \eqref{29-3-23-7-1}, \eqref{29-3-23-1}, \eqref{29-3-23-1-2}, \eqref{29-3-23-2-1} to reach at 
\begin{equation}\label{29-3-23-11}
\begin{aligned}
\mathcal{Z}_{m}\rightarrow-\left(\eta,u\right)_{X_{0}(\Omega)}-\left(\partial^{\alpha}_{t}u,u\right)+(f,u)+\int_{0}^{t}B(t,\tau,u(\tau),u)~d\tau.
\end{aligned}
\end{equation}
Take $v=u$ in equation \eqref{N24a1} to deduce that $\mathcal{Z}_{m}\rightarrow 0~a.e.~t\in (0,T)$. Further, rewrite equation \eqref{29-3-23-9} as
\begin{equation}\label{29-3-23-10-1}
\begin{aligned}
M\left(\|u_{m}(t)\|^{2}_{X_{0}(\Omega)}\right)\|u_{m}-u\|^{2}_{X_{0}(\Omega)}&=\mathcal{Z}_{m}+\int_{0}^{t}B(t,\tau,u_{m}(\tau),u_{m}-u)~d\tau\\
&=\mathcal{Z}_{m}+\int_{0}^{t}(b(x,t,\tau)u_{m}(\tau),u_{m}-u)~d\tau.
\end{aligned}
\end{equation}
Employing \eqref{positivity of diffusion coefficient}, \eqref{8-1-24-1}, and applying Cauchy-Schwarz  inequality, we get
\begin{equation}\label{29-3-23-10}
\begin{aligned}
\|u_{m}-u\|^{2}_{X_{0}(\Omega)}&\lesssim |\mathcal{Z}_{m}| +\|u_{m}(t)-u(t)\|\left(\|(-\Delta)^{s}u_{m}\|_{L^{2}(0,T;L^{2}(\Omega))}+\|u_{m}\|_{L^{2}(0,T;L^{2}(\Omega))}\right).
\end{aligned}
\end{equation}
Apply \eqref{20-4-6}, \eqref{29-3-23-1-2} and letting $m\rightarrow \infty$ in \eqref{29-3-23-10}  to deduce 
\begin{equation}\label{29-3-23-13}
\begin{aligned}
\lim_{m\rightarrow \infty}\|u_{m}-u\|^{2}_{X_{0}(\Omega)}\leq 0 \quad a.e.~t\in (0,T).
\end{aligned}
\end{equation}
Thus 
\begin{equation}\label{29-3-23-15}
\begin{aligned}
\lim_{m\rightarrow \infty}\|u_{m}-u\|_{X_{0}(\Omega)} = 0 \quad a.e.~t\in (0,T).
\end{aligned}
\end{equation}
Using continuity of $M$, we get 
\begin{equation}\label{29-3-23-15a}
\begin{aligned}
M\left(\|u_{m}\|^{2}_{X_{0}(\Omega)}\right)\rightarrow M\left(\|u\|^{2}_{X_{0}(\Omega)}\right)\quad a.e.~t\in (0,T).
\end{aligned}
\end{equation}
This implies 
\begin{equation}\label{29-3-23-16}
\begin{aligned}
\eta= M\left(\|u\|^{2}_{X_{0}(\Omega)}\right)u\quad a.e.~t\in (0,T).
\end{aligned}
\end{equation}
Put \eqref{29-3-23-16} in \eqref{N24a1} to conclude the existence of the weak solution.\\\\
\textbf{(Uniqueness)} Suppose that $u_{1}$ and $u_{2}$ are two solutions of the equation \eqref{weak form 1} then $w=u_{1}-u_{2}$ satisfies the following equations for all $v \in X_{0}(\Omega)$ and $a.e. ~t\in(0,T)$ 
\begin{equation}\label{weak form 1-1-1}
\begin{aligned}
&\left(\partial^{\alpha}_{t}w,v\right)+M\left(\|u_{1}(t)\|^{2}_{X_{0}(\Omega)}\right)(w,v)_{X_{0}(\Omega)}\\
&=\left(M\left(\|u_{2}(t)\|^{2}_{X_{0}(\Omega)}\right)-M\left(\|u_{1}(t)\|^{2}_{X_{0}(\Omega)}\right)\right)(u_{2},v)_{X_{0}(\Omega)}+\int_{0}^{t}B(t,\tau,w(\tau),v)~d\tau.
    \end{aligned}
\end{equation}
 Put  $v=w$ in \eqref{weak form 1-1-1} and apply \eqref{positivity of diffusion coefficient},  \eqref{8-1-24-1} to  get 
 \begin{equation}\label{weak form 1-1}
\begin{aligned}
&\left(\partial^{\alpha}_{t}w,w\right)+m_{0}\|w\|^{2}_{X_{0}(\Omega)}\\
&\leq L_{M}\left(\|u_{2}(t)\|_{X_{0}(\Omega)}+\|u_{1}(t)\|_{X_{0}(\Omega)}\right)\|u_{2}\|_{X_{0}(\Omega)}\|w\|^{2}_{X_{0}(\Omega)}+B_{0}\|w(t)\|_{X_{0}(\Omega)}\int_{0}^{t}\|w(\tau)\|_{X_{0}(\Omega)}~d\tau.
    \end{aligned}
\end{equation}
Using a priori bound \eqref{20-4-6} and Young's inequality we have
\begin{equation}\label{weak form 1-2}
\begin{aligned}
&\left(\partial^{\alpha}_{t}w,w\right)+\left(m_{0}-4L_{M}K^{2}\right)\|w\|^{2}_{X_{0}(\Omega)}\lesssim \int_{0}^{t}\|w(\tau)\|^{2}_{X_{0}(\Omega)}~d\tau.
    \end{aligned}
\end{equation}
Now proceed as we prove estimate \eqref{20-4-1} to conclude $w=0$.\\\\
\textbf{(Initial condition)} Since $u$ is the weak solution of  \eqref{problem statement}, therefore it satisfies  the following equation for all $\phi \in C^{1}([0,T];X_{0}(\Omega))$ such that 
\begin{equation}\label{29-3-23-17}
\begin{aligned}
\int_{0}^{T}\int_{\Omega}\partial^{\alpha}_{t}u\phi~dxdt&+\int_{0}^{T}\int_{\Omega}M\left(\|u(t)\|^{2}_{X_{0}(\Omega)}\right)(-\Delta)^{s}u\phi~dxdt\\
&=\int_{0}^{T}\int_{\Omega}f\phi~dxdt+\int_{0}^{T}\int_{\Omega}\int_{0}^{t}b(x,t,\tau)u(\tau)\phi~d\tau dxdt.
\end{aligned}
\end{equation}
Integrating by parts in time  in the equation \eqref{29-3-23-17} by taking $\phi(T)=0$ we get 
\begin{equation}\label{29-3-23-18}
\begin{aligned}
-\int_{0}^{T}\int_{\Omega}k\ast(u(t)-u_{0})\phi_{t}~dxdt&+\int_{0}^{T}\int_{\Omega}M\left(\|u(t)\|^{2}_{X_{0}(\Omega)}\right)(-\Delta)^{s}u\phi~dxdt\\
&=\int_{0}^{T}\int_{\Omega}f\phi~dxdt+\int_{0}^{T}\int_{\Omega}\int_{0}^{t}b(x,t,\tau)u(\tau)\phi~d\tau dxdt\\
&+\int_{\Omega}(k\ast(u(t)-u_{0}))(0)\phi(0)~dx.
\end{aligned}
\end{equation}
From equation \eqref{N24a}, we also have 
\begin{equation}\label{29-3-23-19}
\begin{aligned}
\int_{0}^{T}\int_{\Omega}\partial^{\alpha}_{t}u_{m}\phi~dxdt&+\int_{0}^{T}\int_{\Omega}M\left(\|u_{m}(t)\|^{2}_{X_{0}(\Omega)}\right)(-\Delta)^{s}u_{m}\phi~dxdt\\
&=\int_{0}^{T}\int_{\Omega}f\phi~dxdt+\int_{0}^{T}\int_{\Omega}\int_{0}^{t}b(x,t,\tau)u_{m}(\tau)\phi~d\tau dxdt.
\end{aligned}
\end{equation}
Again integrating by parts in time  in the equation \eqref{29-3-23-19} by taking $\phi(T)=0$ we get 
\begin{equation}\label{29-3-23-20}
\begin{aligned}
-\int_{0}^{T}\int_{\Omega}k\ast(u_{m}(t)-u_{m}(0))\phi_{t}~dxdt&+\int_{0}^{T}\int_{\Omega}M\left(\|u_{m}(t)\|^{2}_{X_{0}(\Omega)}\right)(-\Delta)^{s}u_{m}\phi~dxdt\\
&=\int_{0}^{T}\int_{\Omega}f\phi~dxdt+\int_{0}^{T}\int_{\Omega}\int_{0}^{t}b(x,t,\tau)u_{m}(\tau)\phi~d\tau dxdt,
\end{aligned}
\end{equation}
where we have used the fact that $(k\ast(u_{m}-u_{m}(0)))(0)=0$. Let $m\rightarrow \infty$ in equation \eqref{29-3-23-20}, we conclude 
\begin{equation}\label{29-3-23-21}
\begin{aligned}
-\int_{0}^{T}\int_{\Omega}k\ast(u(t)-u_{0})\phi_{t}~dxdt&+\int_{0}^{T}\int_{\Omega}M\left(\|u(t)\|^{2}_{X_{0}(\Omega)}\right)(-\Delta)^{s}u\phi~dxdt\\
&=\int_{0}^{T}\int_{\Omega}f\phi~dxdt+\int_{0}^{T}\int_{\Omega}\int_{0}^{t}b(x,t,\tau)u(\tau)\phi~d\tau dxdt.
\end{aligned}
\end{equation}
Comparing equations \eqref{29-3-23-18} and \eqref{29-3-23-21} we obtain 
\begin{equation}
\int_{\Omega}(k\ast(u(t)-u_{0}))(0)\phi(0)~dx=0.
\end{equation}
Since $\phi$ is arbitrary therefore we get $(k\ast(u(t)-u_{0}))(0)=0$. Further, using \cite[Proposition 6.7]{Kubica}, we conclude $u(0)=u_{0}$ for $\alpha \in \left(\frac{1}{2},1\right)$. For other values of $\alpha \in \left(0,\frac{1}{2}\right]$, we need more compatibility conditions on the given data \cite{Kubica}.
 \end{proof}
\section{ Regularity of the  solution of $(\mathcal{K}^{s}_{\alpha})$}\label{7-1-24-3}
In this section, we discuss the regularity of the solution of \eqref{problem statement}. Recall the elliptic regularity results for the homogeneous problem 
\begin{equation}\label{12-1-24-1}
\begin{aligned}
(-\Delta)^{s}u=g\quad \text{in}~\Omega,\\
u=0\quad \text{in}~\Omega^{c}.
\end{aligned}
\end{equation}
\begin{thm}\cite[Proposition 2.1]{acosta2019finite} \label{12-1-24-4} Let $\Omega \subset \mathbb{R}^{d}$ be a bounded domain with smooth boundary. Assume that $g\in L^{2}(\Omega)$ and the solution $u$ of the problem \eqref{12-1-24-1} is in $X_{0}(\Omega)$. Then its solution $u$  satisfies the following regularity estimate 
\begin{equation}\label{12-1-24-2}
\|u\|_{H^{s+\nu}(\Omega)} \lesssim \|g\|_{L^{2}(\Omega)},
\end{equation}
 where $\nu=\min\left\{s,\frac{1}{2}-\epsilon\right\}$ for some small $\epsilon >0$.
\end{thm}

 \begin{thm}\label{29-3-23-22} Under the assumptions (H1)-(H3), we have the following regularity estimate on  the solution $u$ of  \eqref{problem statement}
 \begin{equation}\label{29-3-23-21-1}
 \|u\|_{L^{2}(0,T;H^{s+\nu}(\Omega))} \lesssim \|u_{0}\|_{X_{0}(\Omega)}+\|f\|_{L^{2}_{\alpha}(0,T;L^{2}(\Omega))},
 \end{equation}
  where $\nu=\min\left\{s,\frac{1}{2}-\epsilon\right\}$ for some small $\epsilon >0$.
 \end{thm}
 \begin{proof} The equation \eqref{problem statement} can be rewritten as
 \begin{equation}\label{12-1-24-3}
    (-\Delta)^{s}u=\frac{1}{M\left(\|u(t)\|^{2}_{X_{0}(\Omega)}\right)}\left(-\partial^{\alpha}_{t}u+f(x,t)+ \int_{0}^{t}b(x,t,\tau)u(\tau)~d\tau\right),
\end{equation}
with the homogeneous Dirichlet boundary condition. Using the a priori bounds  \eqref{20-4-1},  \eqref{20-4-6}, \eqref{32-5}, and \eqref{20-4-7} we can deduce that RHS of \eqref{12-1-24-3} belongs to $L^{2}(\Omega)$. Therefore, by applying Theorem \ref{12-1-24-4} we can conclude the desired result \eqref{29-3-23-21-1}.
\end{proof}

\begin{thm}\label{13-1-24-1} Suppose that $u_{0}\in X_{0}(\Omega)$ such that $(-\Delta)^{s}u_{0}\in L^{2}(\Omega)$ and $f= b=0$ in \eqref{problem statement}. Then the solution $u$ of   \eqref{problem statement} satisfies the following estimates 
\begin{align}
\|u(t)\|&\lesssim ~(1+t^{\alpha})\quad \forall~t\in [0,T] \label{12-1-24-16}.\\
\|u(t)-u(t^{\ast})\| &\lesssim ~|t-t^{\ast}|^{\alpha}\quad \forall~t,t^{\ast}\in [0,T]\label{12-1-24-15}.\\
\|\partial^{\alpha}_{t}u\|&\lesssim~ C\quad \forall~t\in [0,T]\label{12-1-24-15A}.
\end{align}
\end{thm}
\begin{proof} Since the solution $u$ of  \eqref{weak form 1} belongs to $L^{2}(\Omega)$ thus we can write $u$ as a Fourier series expansion in terms of eigenfunctions $\{\phi_{i}\}_{i\in \mathbb{N}}$ of the fractional eigenvalue problem  \eqref{20-1-24-1} i.e.,
\begin{equation}\label{12-1-24-5}
u(x,t)=\sum_{i=1}^{\infty}\alpha_{i}(t)\phi_{i}(x).
\end{equation}
By putting the expression of $u$ $\eqref{12-1-24-5}$ in \eqref{weak form 1} along with $f=b=0$ and $v=\phi_{j}$ then we get 
\begin{equation}\label{12-1-24-6}
\begin{aligned}
\partial^{\alpha}_{t}\alpha_{i}(t)&=-\lambda_{i}M\left(\|u(t)\|^{2}_{X_{0}(\Omega)}\right)\alpha_{i}(t)\quad \forall~i=1,2,3,\dots\\
\alpha_{i}(0)&=(u_{0},\phi_{i})\quad \forall~i=1,2,3,\dots.
\end{aligned}
\end{equation}
In \eqref{12-1-24-6}, we have used the properties of eigenfunctions being orthonormal in $L^{2}(\Omega)$ and orthogonal in $X_{0}(\Omega)$. By means of integral transformation \cite{webb2019weakly}, the equation \eqref{12-1-24-6} is transformed into 
\begin{equation}\label{12-1-24-7}
\alpha_{i}(t)=\alpha_{i}(0)-\lambda_{i}\int_{0}^{t}\frac{(t-\tau)^{\alpha-1}}{\Gamma(\alpha)}M\left(\|u(\tau)\|^{2}_{X_{0}(\Omega)}\right)\alpha_{i}(\tau)~d\tau \quad \forall~i=1,2,3,\dots.
\end{equation}
Taking modulus on both sides of \eqref{12-1-24-7} and employing  a priori bound \eqref{20-4-6} and continuity of $M$, we get  $M\left(\|u(\tau)\|^{2}_{X_{0}(\Omega)}\right) \leq C$ and therefore
\begin{equation}\label{12-1-24-8}
|\alpha_{i}(t)|\leq|\alpha_{i}(0)|+\lambda_{i}C\int_{0}^{t}\frac{(t-\tau)^{\alpha-1}}{\Gamma(\alpha)}|\alpha_{i}(\tau)|~d\tau \quad \forall~i=1,2,3,\dots.
\end{equation}
An application of Gronwall's inequality Lemma \ref{Gronwall inequality}, we deduce
\begin{equation}\label{12-1-24-9}
|\alpha_{i}(t)|\leq |\alpha_{i}(0)| E_{\alpha}(\lambda_{i}Ct^{\alpha}) \quad \forall~i=1,2,3,\dots.
\end{equation}
Using the definition of $\alpha_{i}(0)$ and estimate on Mittag-Leffler function \eqref{8-1-24-2}, we obtain 
\begin{equation}\label{12-1-24-10}
|\alpha_{i}(t)|\lesssim |(u_{0},\phi_{i})| \quad \forall~i=1,2,3,\dots.
\end{equation}
Consider \eqref{12-1-24-5}, then we have 
\begin{equation}\label{12-1-24-11}
\|(-\Delta)^{s}u(t)\|^{2}=\sum_{j=1}^{\infty}\lambda_{j}^{2}|\alpha_{j}(t)|^{2} \lesssim \sum_{j=1}^{\infty}\lambda_{j}^{2}|(u_{0},\phi_{i})|^{2}=\|(-\Delta)^{s}u_{0}\|^{2} \quad \forall~t\in [0,T] .
\end{equation}
Now, we take the equation \eqref{problem statement} with $f=b=0$ i.e.,
\begin{equation}\label{12-1-24-12}
 \partial^{\alpha}_{t}u=-M\left(\|u(t)\|^{2}_{X_{0}(\Omega)}\right)(-\Delta)^{s}u.
\end{equation}
Again applying fractional integral operator of order $\alpha$, then the equation \eqref{12-1-24-12} is converted into 
\begin{equation}\label{12-1-24-13}
u(t)=u_{0}-\int_{0}^{t}\frac{(t-\tau)^{\alpha-1}}{\Gamma(\alpha)}M\left(\|u(\tau)\|^{2}_{X_{0}(\Omega)}\right)(-\Delta)^{s}u(\tau)~d\tau.
\end{equation}
Taking $L^{2}(\Omega)$ norm on both side of \eqref{12-1-24-13}, and applying a priori bound \eqref{20-4-6} and \eqref{positivity of diffusion coefficient} along with estimate \eqref{12-1-24-11} to get 
\begin{equation}\label{12-1-24-14}
\|u(t)\|~\lesssim~\left( 1+\int_{0}^{t}\frac{(t-\tau)^{\alpha-1}}{\Gamma(\alpha)}~d\tau\right)~ \lesssim~ (1+t^{\alpha}).
\end{equation}
\par To prove the second estimate \eqref{12-1-24-15}, we first consider $t^{\ast}< t$. Using equation \eqref{12-1-24-13}, we get 
\begin{equation}\label{12-1-24-17}
\begin{aligned}
u(t)-u(t^{\ast})& = \frac{1}{\Gamma(\alpha)}\int_{0}^{t^{\ast}}(t^{\ast}-\tau)^{\alpha-1}M\left(\|u(\tau)\|^{2}_{X_{0}(\Omega)}\right)(-\Delta)^{s}u(\tau)~d\tau\\
&-\frac{1}{\Gamma(\alpha)}\int_{0}^{t}(t-\tau)^{\alpha-1}M\left(\|u(\tau)\|^{2}_{X_{0}(\Omega)}\right)(-\Delta)^{s}u(\tau)~d\tau,
\end{aligned}
\end{equation}
which can also be rewritten as 
\begin{equation}\label{12-1-24-18}
\begin{aligned}
u(t)-u(t^{\ast})& = \frac{1}{\Gamma(\alpha)}\int_{0}^{t^{\ast}}\left[(t^{\ast}-\tau)^{\alpha-1}-(t-\tau)^{\alpha-1}\right]M\left(\|u(\tau)\|^{2}_{X_{0}(\Omega)}\right)(-\Delta)^{s}u(\tau)~d\tau\\
&-\frac{1}{\Gamma(\alpha)}\int_{t^{\ast}}^{t}(t-\tau)^{\alpha-1}M\left(\|u(\tau)\|^{2}_{X_{0}(\Omega)}\right)(-\Delta)^{s}u(\tau)~d\tau.
\end{aligned}
\end{equation}
Thus, we have 
\begin{equation}\label{12-1-24-19}
\begin{aligned}
\|u(t)-u(t^{\ast})\|& \leq  \frac{1}{\Gamma(\alpha)}\int_{0}^{t^{\ast}}|(t^{\ast}-\tau)^{\alpha-1}-(t-\tau)^{\alpha-1}|~|M\left(\|u(\tau)\|^{2}_{X_{0}(\Omega)}\right)|~\|(-\Delta)^{s}u(\tau)\|~d\tau\\
&+\frac{1}{\Gamma(\alpha)}\int_{t^{\ast}}^{t}|(t-\tau)^{\alpha-1}|~|M\left(\|u(\tau)\|^{2}_{X_{0}(\Omega)}\right)|~\|(-\Delta)^{s}u(\tau)\|~d\tau.
\end{aligned}
\end{equation}
Apply a priori bound \eqref{20-4-6} and \eqref{positivity of diffusion coefficient} along with estimate \eqref{12-1-24-11} to get 
\begin{equation}\label{12-1-24-20}
\begin{aligned}
\|u(t)-u(t^{\ast})\|& \lesssim \int_{0}^{t^{\ast}}|(t^{\ast}-\tau)^{\alpha-1}-(t-\tau)^{\alpha-1}|~d\tau+\int_{t^{\ast}}^{t}|(t-\tau)^{\alpha-1}|~d\tau.
\end{aligned}
\end{equation}
As we have $t^{\ast}<t$ this implies  $t^{\ast}-\tau < t-\tau$ for $\tau \in (0,t^{\ast})$. Hence $(t^{\ast}-\tau)^{\alpha-1} >(t-\tau)^{\alpha-1}$. Therefore 
\begin{equation}\label{12-1-24-21}
\begin{aligned}
\|u(t)-u(t^{\ast})\|& \lesssim \int_{0}^{t^{\ast}}\left((t^{\ast}-\tau)^{\alpha-1}-(t-\tau)^{\alpha-1}\right)~d\tau+\int_{t^{\ast}}^{t}(t-\tau)^{\alpha-1}~d\tau\\
&=\left[(t-t^{\ast})^{\alpha}+(t^{\ast})^{\alpha}-t^{\alpha}\right] +(t-t^{\ast})^{\alpha}\lesssim (t-t^{\ast})^{\alpha}.
\end{aligned}
\end{equation}
By interchanging the role of $t$ and $t^{\ast}$, we have
\begin{equation}\label{29-3-23-25}
\|u(t^{\ast})-u(t)\| \lesssim (t^{\ast}-t)^{\alpha}\quad \text{for}~t^{\ast}>t.
\end{equation}
Hence, the estimate \eqref{12-1-24-15} follows immediately using \eqref{12-1-24-21}-\eqref{29-3-23-25}. We can easily prove estimate \eqref{12-1-24-15A} using equation \eqref{12-1-24-12}.
\end{proof}
 \begin{rmk}
 For the nonzero values of  $f$ and $b$ in \eqref{problem statement}, we need to take extra regularity on $f$ i.e., $\|(-\Delta)^{s}f(t)\|$ is bounded for every $t\in [0,T]$. Then it is easy to prove  estimates  \eqref{12-1-24-16}-\eqref{12-1-24-15A} by applying the same techniques as in Theorem \ref{13-1-24-1}.
 \end{rmk}
 \begin{rmk}
     Similar  analysis of this paper can be done to study the space-time fractional partial integro-differential equations with  nonlocal diffusion coefficient $M\left(\|u\|^{2}_{Z}\right)$ where $Z$ is equal to $L^{2}(\Omega)$ or $L^{1}(\Omega)$.
 \end{rmk}

 \bibliographystyle{plain}

\begin{thebibliography}{10}
 
 \bibitem{barbeiro2011h1}{S. Barbeiro, J. A. Ferreira, and  L. Pinto. H1-second order convergent estimates for non-Fickian models. Applied numerical mathematics, 61(2), 201--215, 2011.      }

 \bibitem{li2018some}{ L. Li and J. G. Liu. Some compactness criteria for weak solutions of time-fractional PDEs. SIAM Journal on Mathematical Analysis, 50(4), 3963--3995, 2018.   }
  
\bibitem{podlubny1998fractional}{I. Podlubny. Fractional differential equations: an introduction to fractional derivatives, fractional differential equations, to methods of their solution and some of their application. Elsevier, 1998. }
  
\bibitem{kumar2020finite}{L. Kumar, S. G. Sista, and K. Sreenadh. Finite element analysis of parabolic integro-differential equations of Kirchhoff type. Mathematical Methods in the Applied Sciences, 43(15), 9129--9150, 2020.}
  
\bibitem{Kubica}{A. Kubica and M. Yamamoto. Initial-boundary value problem for fractional diffusion equations with time-dependent coefficients. Fractional Calculus and Applied Analysis, 21(2), 112--125, 2018.  }
 
\bibitem{zacher2009weak}{ R. Zacher. Weak solutions of abstract evolutionary integro-differential equations in Hilbert spaces. Funkcialaj Ekvacioj, 52(1), 1--18, 2009.    }
  
\bibitem{almeida2017gronwall}{ R. Almeida. A Gronwall inequality for a general Caputo fractional operator. Mathematical Inequalities and Application, 20(4), 1089--1105, 2017.    }
	

\bibitem{kirchhoff1883vorlesungen}{G. Kirchhoff. Vorlesungen uber. Mechanik, Leipzig, Teubner, 1883. }
 
\bibitem{chipot2003remarks}  M. Chipot, V. Valente and  G. C. Vergara. Remarks on a nonlocal problem involving the Dirichlet energy. Rendiconti del Seminario Matematico della Universit{\`a} di Padova, 110, 199--220, 2003. 
 

\bibitem{diethelm2002analysis} K. Diethelm and  N. J. Ford. Analysis of fractional differential equations. Journal of Mathematical Analysis and Applications. 265(2), 229--248, 2002.  


\bibitem{webb2019weakly}J. R . L. Webb. Weakly singular Gronwall inequalities and applications to fractional differential equations. Journal of Mathematical Analysis and Applications, 471(1-2), 692--711, 2019. 
  

\bibitem{mingqi2018nonlocal}  X. Mingqi, V. D. R{\u{a}}dulescu, and  B. Zhang. Nonlocal Kirchhoff diffusion problems: local existence and blow-up of solutions, Nonlinearity, 31(7), 3228,  2018.
  

\bibitem{di2012hitchhiker} E. D. Nezza, G. Palatucci, and  E. Valdinoci. Hitchhiker guide to the fractional Sobolev spaces. Bulletin des sciences math{\'e}matiques, 136(5), 521--573, 2012.   

\bibitem{cuesta2012image} E. Cuesta, M. Kirane, and  S. A. Malik. Image structure preserving denoising using generalized fractional time integrals. Signal Processing, 92(2), 553--563, 2012.  


\bibitem{applebaum2004levy}D. Applebaum. L{\'e}vy processes-from probability to finance and quantum groups. Notices of the AMS, 51(11), 1336--1347, 2004.  

\bibitem{luchko2012anomalous} Y. Luchko. Anomalous diffusion: models, their analysis, and interpretation. Advances in applied analysis, 115--145, 2012. 
 

\bibitem{del2004fractional} D. D. C. Negrete, B. A. Carreras, and V. E. Lynch. Fractional diffusion in plasma turbulence. Physics of Plasmas, 11(8), 3854--3864, 2004.    
 

\bibitem{biler2015nonlocal} P. Biler,  C. Imbert, and  G. Karch. The nonlocal porous medium equation: Barenblatt profiles and other weak solutions, Archive for Rational Mechanics and Analysis, 215, 497--529, 2015.
  

\bibitem{miller1978integrodifferential} R. K. Miller. An integro-differential equation for rigid heat conductors with memory. Journal of Mathematical Analysis and Applications, 66(2), 313--332, 1978. 
 
\bibitem{kumar2023linearized} L. Kumar, S. G. Sista, and  K. Sreenadh. A Linearized L1-Galerkin FEM for Non-smooth Solutions of Kirchhoff Type Quasilinear Time-Fractional Integro-Differential Equation, Journal of Scientific Computing, 96(2), 36, 2023.  
  

\bibitem{kumar2021finite}  L. Kumar, S. G. Sista, and  K. Sreenadh. Finite Element Analysis of Time Fractional Integro-differential Equations of Kirchhoff type for Non-homogeneous Materials. Mathematical Methods in the Applied Sciences, 47(4), 2120--2153, 2024.


\bibitem{shen2022time} R. Shen, X. Mingqi, and V. D. R{\u{a}}dulescu. Time-space fractional diffusion problems: Existence, decay estimates and blow-up of solutions. Milan Journal of Mathematics, 90(1), 103--129, 2022.     


\bibitem{fu2022global}  Y. Fu and X. Zhang. Global existence and asymptotic behavior of weak solutions for time-space fractional Kirchhoff-type diffusion equations. Discrete \& Continuous Dynamical Systems-Series B, 27(3), 2022. 
  
\bibitem{servadei2013lewy}  R. Servadei and E. Valdinoci. Lewy--Stampacchia type estimates for variational inequalities driven by nonlocal operators. Revista Matem{\'a}tica Iberoamericana, 29(3), 1091--1126, 2013.  


\bibitem{servadei2012mountain} R. Servadei and  E. Valdinoci.  Mountain pass solutions for nonlocal elliptic operators. Journal of Mathematical Analysis and Applications, 389(2), 887--898, 2012.

\bibitem{Giacomoni} J. Giacomoni, T. Mukherjee, and K. Sreenadh. Existence and stabilization results for a singular parabolic equation involving the fractional Laplacian. Discrete \& Continuous Dynamical Systems-Series S, 12(2), 311–337, 2019.
 

\bibitem{rawat2021multiple} S. Rawat  and   K. Sreenadh. Multiple positive solutions for degenerate Kirchhoff equations with singular and Choquard nonlinearity. Mathematical Methods in the Applied Sciences, 44(18), 13812--13832, 2021. 
 

\bibitem{gripenberg1990volterra} G. Gripenberg, S. O. Londen, and O. Staffans. Volterra integral and functional equations, Cambridge University Press, 34, 1990.    


\bibitem{pedregal1997parametrized}P. Pedregal. Parametrized measures and variational principles. Springer Science \& Business Media, 1997.


\bibitem{servadei2013variational} R. Servadei  and  E. Valdinoci. Variational methods for nonlocal operators of elliptic type. Discrete \& Continuous Dynamical Systems, 33(5), 2105--2137,   2013. 


\bibitem{acosta2019finite} G. Acosta, F. M. Bersetche, and   J. P. Borthagaray. Finite element approximations for fractional evolution problems. Fractional Calculus and Applied Analysis, 22(3), 767--794, 2019. 
  

\bibitem{ding2020local}  H. Ding and J. Zhou. Local existence, global existence and blow-up of solutions to a nonlocal Kirchhoff diffusion problem. Nonlinearity, 33(3), 1046, 2020. 
  

\bibitem{binlin2019existence} B. Zhang,  V. D. R{\u{a}}dulescu, and L. Wang. Existence results for Kirchhoff--type superlinear problems involving the fractional Laplacian. Proceedings of the Royal Society of Edinburgh Section A: Mathematics, 149(4), 1061--1081, 2019.
 

\bibitem{goel2022critical} D. Goel, S. Rawat, and  K. Sreenadh. Critical growth fractional Kirchhoff elliptic problems. arXiv preprint arXiv:2203.06471, 2022.  
  
  \end{thebibliography}

\end{document}